\documentclass{amsart}
\usepackage{lmodern}
\usepackage{microtype} 
\usepackage{amsthm,amsfonts,amssymb,amsmath,thmtools,interval}
\usepackage{bm}%\boldsymbol
\usepackage{calligra}%mathcal
\usepackage{mathrsfs}%mathscr
\usepackage{mathtools}
\usepackage{stmaryrd}%mapsfrom
\usepackage{tikz}%picture
\usepackage{tikz-3dplot}
\usepackage{tikz-cd}%commutative-diagram
\usepackage[all]{xy}

\usepackage{color}
\definecolor{dkgreen}{rgb}{0,0.4,0}
\definecolor{mblue}{rgb}{0,0.5,0.62}
\definecolor{gray}{rgb}{0.5,0.5,0.5}
\definecolor{burgendy}{rgb}{0.502, 0, 0.125}

\usepackage{hyperref}%ref
\hypersetup{
    linktoc=page,
        colorlinks = true,
        linkcolor = mblue, %Colour of internal links
        filecolor = cyan,
        urlcolor = dkgreen,
        citecolor = burgendy, %Colour of citations
}

\usepackage[capitalize]{cleveref}

\theoremstyle{plain}
\newtheorem{thm}{Theorem}[subsection]
\newtheorem{cor}[thm]{Corollary}
\newtheorem{lem}[thm]{Lemma}
\newtheorem{prop}[thm]{Proposition}
\newtheorem{conj}[thm]{Conjecture}

\theoremstyle{definition}
\newtheorem{defn}[thm]{Definition}
\newtheorem{rmk}[thm]{Remark}
\newtheorem{exe}[thm]{Example}

\makeatletter
    \let\c@equation\c@thm
\makeatother
\numberwithin{equation}{subsection}

\newcommand{\bb}[1]{\mathbb{#1}}
\newcommand{\rr}[1]{\mathrm{#1}}
\newcommand{\cc}[1]{\mathcal{#1}}
\newcommand{\scr}[1]{\mathscr{#1}}
\newcommand{\fk}[1]{\mathfrak{#1}}

\newcommand{\cros}{^{\times}}
\newcommand{\dR}{\mathrm{dR}}

\newcommand{\gr}{\mathrm{gr}}

\newcommand{\im}{\mathrm{im}}
\newcommand{\inv}{^{-1}}
\newcommand{\pr}{\mathrm{pr}}

\newcommand{\Hodge}{^{\mathrm{H}}}

\let\emptyset\varnothing

\title{Classical and irregular Hodge numbers}

\author{Yichen Qin } 
\address{School of Mathematical Sciences, Fudan University, 
     Handan Road 220, 200433 Shanghai, China}
\email{yichenqin@fudan.edu.cn}

\author{Dingxin Zhang} 
\address{Center for Mathematics and Interdisciplinary Sciences, Fudan University, and Shanghai Institute for Mathematics and Interdisciplinary Sciences,  200433 Shanghai, China}
\email{zhang@simis.cn}

\begin{document}

\begin{abstract}
Let $U$ be a smooth quasi-projective complex variety with a regular function $f$. The twisted de Rham cohomology groups $\rr{H}^k_{\mathrm{dR}}(U, f)$ carry the decreasing irregular Hodge filtration, whose graded pieces have dimensions known as the irregular Hodge numbers. In this paper, we prove that the irregular Hodge numbers admit an explicit characterization in terms of classical Hodge numbers, closely related to Hodge-theoretic numbers constructed by Katzarkov, Kontsevich, and Pantev for Landau--Ginzburg models. As direct applications, we show that irregular Hodge numbers of non-degenerate functions are independent of the choice of non-degenerate functions,  and we give a concrete formula for irregular Hodge numbers for unipotent non-degenerate functions.
\end{abstract}

\maketitle

%\tableofcontents

\section{Introduction}

Let \(U\) be a smooth quasi-projective variety of pure dimension \(n\)
over the complex numbers \(\mathbb{C}\), and let
\(f\colon U \to \mathbb{A}^1\) be a regular function. We can   form
the \emph{twisted de~Rham complex}
\[
    \Omega^0_{U/\mathbb{C}} \xrightarrow{\mathrm{d} + \mathrm{d} f}
    \Omega^1_{U/\mathbb{C}} \xrightarrow{\mathrm{d}+\mathrm{d}f}
    \cdots \xrightarrow{\mathrm{d}+\mathrm{d}f} \Omega^{n}_{U/\mathbb{C}},
\]
whose hypercohomology  is called the \emph{twisted de~Rham cohomology},
denoted by
\[
    \mathrm{H}^{\ast}_{\rr{dR}}(U,f) \coloneq \mathbb{H}^{\ast}(U; (\Omega_{U/\mathbb{C}}^{\bullet}, \mathrm{d}+\mathrm{d}f)).
\]
There is also a compactly supported variant. In particular, when $f$ is a constant function, we recover the usual algebraic de Rham cohomologies of $U$. 

Guided by analogies with exponential sums over finite
fields and the weight theory for $\ell$-adic cohomologies, Deligne envisioned \cite{DeligneTheorieHodge07} that there should exist an
\emph{irregular Hodge theory} for these twisted de Rham cohomologies. 
Due to the irregular singularity at infinity of the connection $(\mathscr{O}_U,\mathrm{d}+\mathrm{d}f)$, the twisted de Rham cohomologies \(\mathrm{H}^{\ast}_{\rr{dR}}(U,f)\) cannot be directly studied through the classical Hodge theory or Saito's mixed Hodge modules.

Deligne's vision was realized in a series of works by  Esnault, Kontsevich, Sabbah, Saito and Yu, see  \cite{EsnaultEtAl$E_1$degenerationirregular17,SabbahIrregularHodge18} for example.  As a result, one can now endow
\(\mathrm{H}_{\rr{dR}}^{i}(U,f)\) with a decreasing \textit{irregular Hodge filtration}
\(F_{\rr{irr}}^\bullet\) indexed by rational numbers \(\alpha \in \mathbb{Q}\), yielding a collection of numerical invariants  
\[
    h^{\gamma}_{\rr{irr}}(U,f,i) = \dim \operatorname{gr}_{F_{\mathrm{irr}}}^{\gamma}\mathrm{H}^i_{\rr{dR}}(U,f)
    \quad (\gamma \in \mathbb{Q})
\]
called the \emph{irregular Hodge numbers}.   

\subsection{Main result}
As the classical Hodge numbers are fundamental algebraic
invariants of algebraic varieties, the irregular Hodge numbers should be viewed as the natural algebraic invariants of the pair \((U,f)\). Roughly speaking, these numbers can be thought of as numerical invariants that capture the singularity structure of the ``fiber at
infinity'' \(g^{-1}(\infty)\) in a compactification
\[
    \begin{tikzcd}
        U \ar[r,"j"] \ar[d,"f"] & X \ar[d,"g"] \\
        \mathbb{A}^1 \ar[r] & \mathbb{P}^1,
    \end{tikzcd}
\]
where \(j\) is an open immersion and \(X\) is a smooth projective variety.
 
The purpose of this paper is to report our quantitative understanding of the above idea.  We explain how to link the irregular Hodge numbers of a regular function \(f\) with a certain limiting Hodge numbers well-understood in classical Hodge theory.  To put the story into perspective, we first recall a ``topological''  interpretation of the  twisted de Rham cohomology, due to Deligne--Malgrange~\cite{DeligneEtAlSingularitesirregulieres07}, \makebox{Dimca--Saito}~\cite[Prop.~2.8]{dimca-saito_cohomology-general-fiber}, and Sabbah~\cite[Thm.~1.1]{sabbah_comparison-elementary-irregular}. 

Let \(f\colon U \to \mathbb{A}^1\) be a regular function on a purely
    \(n\)-dimensional smooth algebraic variety over \(\mathbb{C}\).
    Then for all but finitely many \(t\in \mathbb{C}\), we have an
    isomorphism of \(\mathbb{C}\)-vector spaces
    \[
        \mathrm{H}^{i}(U^{\mathrm{an}}, f^{-1}(t)^{\mathrm{an}})
        \cong \mathrm{H}^i_{\rr{dR}}(U,  f)
    \]
    for all $i\in \bb{Z}$, where \(\mathrm{H}^{\ast}(U^{\mathrm{an}},f^{-1}(t)^{\mathrm{an}})\) denotes
    the singular cohomology of the complex manifold \(U^{\mathrm{an}}\)
    relative to the fiber \(f^{-1}(t)^{\mathrm{an}}\), with complex
    coefficients. 
    
    The relative cohomology spaces appearing on the left-hand side of the above identity carry mixed Hodge structures.  As \(t\) varies along $\infty$, there is a quasi-unipotent monodromy action \(T\) on $\mathrm{H}^{i}(U^{\mathrm{an}}, f^{-1}(t)^{\mathrm{an}})$, as well as its decomposition into generalized eigenspaces
\[\mathrm{H}^{i}(U^{\mathrm{an}}, f^{-1}(t)^{\mathrm{an}}) = \bigoplus_{\lambda} \mathrm{H}^{i}(U^{\mathrm{an}}, f^{-1}(t)^{\mathrm{an}})_\lambda,\]
where \(\lambda\) runs over the eigenvalues of \(T\).  A natural Hodge-theoretic construction in this setting is the associated  \emph{limiting mixed Hodge structure} 
    $$(\mathrm{H}^{i}(U^{\mathrm{an}}, f^{-1}(t)^{\mathrm{an}}),F^\bullet_{\rr{lim}},W_\bullet^{\rr{lim}}),$$ 
which respects the above eigenspace decomposition.

Our main result establishes an identity between the limiting
Hodge numbers  and the irregular Hodge numbers.
 
\begin{thm}[{\ref{thm:KKP-numbers}}]\label{theorem:main} 
    With notation as above, for any $\alpha$ in $[0,1) $ and
    \(p \in \mathbb{Z}\), we have
    \[
        h^{p+\alpha}_{\rr{irr}}(U,f,i) = \dim\rr{gr}_{F_{\rr{lim}}}^p\mathrm{H}^{i}(U^{\mathrm{an}}, f^{-1}(t)^{\mathrm{an}})_\lambda,
    \]
    where $\lambda=\exp(-2\pi\mathrm{i}\alpha)$.
\end{thm}

Using Theorem~\ref{theorem:main}, we will provide an explicit formula \cref{proposition:unipotent-computation} for the irregular Hodge numbers for a class of regular functions with unipotent monodromy, relying on an explicit nearby cycle computation.

\subsection{Hodge numbers of Landau--Ginzburg models}
In mirror symmetry, the mirrors of  Fano varieties are expected to be \textit{Landau--Ginzburg models}. Roughly speaking, they are pairs $(U,f)$, consisting of non-proper varieties $U$ and regular functions $f$ on $U$ such that the pole divisor of $f$ being reduced. To get a similar relation on the Hodge diamond of a Fano variety with its mirror, Katzarkov, Kontsevich, and Pantev proposed in \cite[\S3]{KatzarkovEtAlBogomolovTian17} some numerics as candidates for the Hodge numbers of Landau--Ginzburg models, namely $h_{\rr{LG}}^{p,q}(U,f)$, $f^{p,q}_{\rr{LG}}(U,f)$, and $i^{p,q}_{\rr{LG}}(U,f)$, and conjectured that they are the same \cite[Conj. 3.6]{KatzarkovEtAlBogomolovTian17}. 

\begin{rmk}
    Katzarkov, Kontsevich, and Pantev also conjectured in \cite[Conj. 3.7]{KatzarkovEtAlBogomolovTian17} the numerical relation of the Hodge diamonds of Fano varieties and their Landau--Ginzburg models as
        $$f^{p,q}_{\rr{LG}}(U,f) = h^{p,\dim X-q}(X),$$
    which was verified for many cases such as the del Pezzo surfaces, Fano threefolds, toric complete intersections, and minuscule flag varieties in \cite{LUNTS2018189,Harder21,CP25,harder_lee_irregular_hodge_2025,QSS2024Irregular}.
\end{rmk} 

When \(f\) has unipotent monodromy around \(\infty\), we have \(h^{\alpha}_{\rr{irr}}(U,f,i) = 0\) unless \(\alpha \in \mathbb{Z}\); that is,
there are no fractional irregular Hodge numbers in this case. Thanks to the work of Esnault, Sabbah, and Yu in \cite{EsnaultEtAl$E_1$degenerationirregular17}, the numbers $f^{p,q}_{\rr{LG}}(U,f)$ are equal to the irregular Hodge numbers $h_{\rr{irr}}^{p}(U,f,p+q)$. On the other hand, the numbers $i^{p,q}_{\rr{LG}}(U,f)$ should be wrongly stated in the conjecture, as this can be seen by looking at the case of the mirror of projective spaces. As for the remaining  Landau--Ginzburg Hodge numbers, 
\[
    h^{p,q}_{\rr{LG}}(U,f) \coloneq
    \dim\operatorname{gr}^{W^{\rr{lim}}}_{2p} \rr{H}^{p+q}(U^{\mathrm{an}}, f^{-1}(t)^{\mathrm{an}}),
\]
previous works of \cite{Shamoto2018,sabbah2018properties, BGRSL23} showed that
\[
    h^{p,q}_{\rr{LG}}(U,f) =f^{p,q}_{\rr{LG}}(U,f) 
\]
for cases such as  convenient non-degenerate Laurent polynomials. Moreover, Shamoto  proved that when $f$ is weakly tame, the above
equality holds if and only if the limiting Hodge structure on $\rr{H}^{p+q}(U^{\mathrm{an}}, f^{-1}(t)^{\mathrm{an}})$
is of Hodge--Tate type.  

Our theorem \cref{theorem:main} provides a proof of Shamoto's theorem in general, and gives the precise equality that should hold in the unipotent case. 
\begin{cor}[{\ref{prop:kkp}}]\label{cor:KKP-numbers}
    The numbers $h^{p,q}_{\rr{LG}}(U,f)$ and $f^{p,q}_{\rr{LG}}(U,f)$ are equal if and only if the limiting Hodge structure on $\rr{H}^{p+q}(U^{\mathrm{an}}, f^{-1}(t)^{\mathrm{an}})$ is of Hodge--Tate type. 
\end{cor}

\begin{exe}\label{example:counterexample}
    Let $U=\mathbb{P}^2-Z(x_0^3+x_1^3+x_2^3+3x_0x_1x_2)$ and 
        $$f=\frac{x_0x_1x_2}{x_0^3+x_1^3+x_2^3+3x_0x_1x_2}.$$ 
    When $i=1,3$, the cohomology groups $\rr{H}^{i}(U^{\mathrm{an}}, f^{-1}(t)^{\mathrm{an}})$ are $2$-dimensional, and the corresponding limiting Hodge structures are not of Hodge--Tate type.
    However, the twisted de Rham cohomologies $\rr{H}^i_{\dR}(U,f)$ are $0$ unless $i=2$, and the corresponding irregular Hodge numbers of $\rr{H}^2_{\dR}(U,f)$ will be calculated in \cref{example:calculation2}.
\end{exe}

Furthermore, using the formula in \cref{proposition:unipotent-computation}, one can find many examples  in the ``maximally degenerating'' case such that the Katzarkov--Kontsevich--Pantev conjecture is valid.

\subsection{Deformation invariance of irregular Hodge numbers}
Another  result of this paper is the deformation invariance of irregular Hodge numbers for non-degenerate functions. In classical Hodge theory, the Hodge numbers of a smooth projective variety remain unchanged under deformation. 

In the irregular setting, the partial analogue of the geometric condition ``smooth projective'' is that the function \(f\) is ``non-degenerate'' as defined by Katz~\cite[\S4]{KatzSommesexponentielles80} and Mochizuki~\cite[Def. 2.6]{MochizukiTwistorGKZ15}, see \cref{defn:non-degenerate}. Here, we consider triples \((X, D, f)\), where \(X\) is a smooth projective variety, \(D\) is a simple normal crossing divisor on \(X\), and \(f\colon X -D \to \mathbb{A}^1\) is a regular function satisfying a certain condition imposed by Katz.  We will show that the irregular Hodge numbers remain constant in families of non-degenerate functions.

\begin{thm}\label{intro::non-degenerate}
    Let \(X\) be a smooth proper algebraic variety, and let \(D \subset X\) be a simple normal crossing divisor.  Denote by \(U\) the complement of the support of \(D\) in \(X\). For any two non-degenerate functions \(f\) and \(g\) on \((X, D)\), the irregular Hodge numbers of \((U, f)\) and \((U, g)\) coincide.
\end{thm}

\subsection{Organization}

The structure of this article is as follows.
\cref{section:exp-mhm} collects background on Saito's mixed Hodge modules,  and exponential mixed Hodge structures. In particular, the information of a twisted de Rham cohomology with its associated irregular Hodge filtration can be packaged into an exponential mixed Hodge structure.

In \cref{section:stationary-phase}, we establish a modified version of Sabbah--Yu's stationary phase formula for exponential mixed Hodge structures. This formula is important for \cref{intro::non-degenerate}, as it relates the irregular Hodge numbers to the limit mixed Hodge structure at infinity.

Building on the above inputs, in \cref{section:KKP-numbers} we prove \cref{theorem:main}, Then, we recall the conjecture of Katzarkov--Kontsevich--Pantev \cref{conj:KKP} and discuss the relation between Landau--Ginzburg Hodge numbers and the irregular Hodge numbers.

\cref{section:snc-pairs} introduces non-degenerate functions on simple normal crossing pairs following Katz and Mochizuki. They are analogies of smooth projective varieties in the irregular setting. We compare them with compactified tame Landau--Ginzburg models and other non-degenerate conditions in \cref{subsection:non-degenerate-functions}. Then, we prove basic purity and vanishing statements for twisted de Rham cohomologies of non-degenerate functions in \cref{subsection:vanishing}.  To finish, we prove \cref{intro::non-degenerate} in \cref{sec:proof-of-main} using some version of nilpotent orbit theorem.

In the end, we provide an explicit formula of irregular Hodge numbers for strongly non-degenerate functions and give examples in \cref{sec:explicit-formula}.

\section{Preliminaries}\label{section:exp-mhm}
In this section, we recall some background on Saito's mixed Hodge modules and exponential mixed Hodge structures.

\subsection{Mixed Hodge modules}
We recall the basic facts about Saito's mixed Hodge modules that are used in this paper; see \cite{SaitoMixedHodge90} for details.
For any complex algebraic variety $X$, there is an abelian category $\mathrm{MHM}(X)$ of algebraic mixed Hodge modules. Roughly speaking, an object of $\mathrm{MHM}(X)$ consists of the data of an underlying rational perverse sheaf, a regular holonomic filtered $\mathscr{D}_X$-module $(\mathcal{M},F^\bullet)$, and a weight filtration; these pieces satisfy some constraints ensuring that the construction is compatible with the classical Hodge theory. In particular, the category  $\mathrm{MHM}(X)$ is equipped with the faithfully exact functor 
    $$\mathrm{rat}\colon \mathrm{MHM}(X)\to\mathrm{Perv}(X;\mathbb{Q})$$
sending a mixed Hodge module to its underlying rational perverse sheaf. 

When $X$ is a point, the category of mixed Hodge modules coincide with the category of polarized mixed Hodge structures. Over a general base $X$, smooth mixed Hodge modules on $X$ are admissible variation of mixed Hodge structures. Moreover, since each mixed Hodge module on $X$ is smooth on a dense open subset $U$ of an irreducible component of $X$, hence, its restriction to $U$ is an admissible variation of mixed Hodge structures on $U$.

There is also the six-functor formalism for mixed Hodge modules.  For any algebraic morphism $f$ one has the derived functors
\[
    f_*,\; f_!,\; f^*,\; f^!,\; \otimes,\; \mathbb{D},
\]
defined between the bounded derived category of mixed Hodge modules. These operations lift those on perverse sheaves, i.e., these functors commute with the rational realization functor
    $$\mathrm{rat}\colon \rr{D}^b\mathrm{MHM}(X)\to \rr{D}^b_c(X;\mathbb{Q}).$$
In this paper, we denote by $\bb{Q}_X^{\rr{H}}$ to be the (shifted) constant Hodge module, whose  underlying perverse sheaf  is $\bb{Q}_X[\dim X]$. Hence, the cohomology of $\bb{Q}_X^{\rr{H}}$ is the classical singular cohomology of $X$ with its canonical mixed Hodge structure:
    $$\rr{H}^i(X,\bb{Q}_X^{\rr{H}})=\rr{H}^{i+\dim X}(X,\bb{Q}).$$

For a regular function $g$ one has the nearby and vanishing cycle functors
    $$\psi_g,\phi_g\colon   \mathrm{MHM}(X)\to  \mathrm{MHM}(g^{-1}(0)).$$
They are constructed on the filtered $\mathscr{D}$-module side via the $V$-filtration, commute with the rational realization functor
$\mathrm{rat}$. In particular, these functors carry the monodromy action together with its decomposition into generalized eigenspaces $\psi_gM=\psi_{g,1}M\oplus \psi_{g,\neq 1}M$ and $\phi_gM=\phi_{g,1}M\oplus \psi_{g,\neq 1}M$ respectively.

\subsection{Rees modules}
The underlying filtered $\mathscr{D}$-module $(\mathcal{M},F^\bullet)$ of a mixed Hodge module $M$ can be encoded in its Rees module 
\begin{equation}\label{eq:Rees-module}
     R_FM:=\bigoplus_{p\geq 0} z^p F^p\mathcal{M}
\end{equation}
as a filtered $\mathscr{D}_X[z]$-module, where $z$ is a formal variable.  The Rees module $R_FM$ is a coherent $\mathscr{D}_X[z]$-module, and the filtration on $\mathcal{M}$ is recovered from $R_FM$ by taking the associated graded pieces. For Rees modules, there are similar notions of coherence, regularity, holonomicity, strict specializability. So we can do similar constructions to Rees modules as for $\mathscr{D}$-modules, see  \cite{MHMProject} for more details.

\subsection{Exponential mixed Hodge structures}
\textit{Exponential mixed Hodge structures} are mixed Hodge modules on the affine line with no cohomologies, as introduced by Kontsevich and Soibelman in \cite{KontsevichSoibelmanCohomologicalHall11}. We recall some basic properties of exponential mixed Hodge structures following \cite[Appx.]{FresanEtAlHodgetheory22}. Concretely, let $\pi\colon \mathbb{A}^1\to\rr{Spec}\,\mathbb{C}$ be the structure morphism. Define the full subcategory
\[
\mathrm{EMHS}\subset \mathrm{MHM}(\mathbb{A}^1)
\]
consisting of those $\mathcal{M}\in\mathrm{MHM}(\mathbb{A}^1)$ with $\pi_*\mathcal{M}=0$. There is an exact idempotent endofunctor
\begin{equation}\label{eq:Pi}
    \Pi\colon \mathrm{MHM}(\mathbb{A}^1)\to\mathrm{MHM}(\mathbb{A}^1),
    \qquad
    \Pi(\mathcal{M}) \;:=\; \operatorname{sum}_*\bigl(j_!\mathbb{Q}_{\mathbb{G}_m}^{\mathrm{H}}\boxtimes\mathcal{M}\bigr),
\end{equation}
where $j\colon\mathbb{G}_m\hookrightarrow\mathbb{A}^1$ is the inclusion and $\operatorname{sum}\colon\mathbb{A}^1\times\mathbb{A}^1\to\mathbb{A}^1$ is addition. The functor $\Pi$ is an exact projector onto $\mathrm{EMHS}$: its essential image equals $\mathrm{EMHS}$, and $\Pi(\mathcal{N})=\mathcal{N}$ for $\mathcal{N}\in\mathrm{EMHS}$. In particular, any constant Hodge modules on $\mathbb{A}^1$ is annihilated by $\Pi$.

Each exponential mixed Hodge structure $M$ has a weight filtration induced from that of $\mathrm{MHM}(\mathbb{A}^1)$, defined as
    $$W_k^{\rr{EMHS}} M:=\Pi(W_k^{\mathrm{MHM}} M),$$
where $W_\bullet^{\mathrm{MHM}}$ is the weight filtration in the category $\mathrm{MHM}(\mathbb{A}^1)$. In general, the weight filtration $W^{\rr{EMHS}}_\bullet$ is different from the weight filtration in $\mathrm{MHM}(\mathbb{A}^1)$. Unless otherwise specified, the weight filtration $W_\bullet$ on an exponential mixed Hodge structure $M$ in this paper always refers to $W_\bullet^{\rr{EMHS}}$.

The de Rham fiber functor is an exact functor
    $$\Xi_{\rr{dR}}\colon \mathrm{EMHS}\to \rr{Vect}_{\bb{C}},
    \quad 
    \rr{H}^1_{\rr{dR}}(\bb{A}^1,M\otimes \mathcal{E}^t)$$ 
where $\mathcal{E}^t$ is the exponential $\scr{D}_{\bb{A}^1}$-module $(\cc{O},\rr{d}+\rr{d}t)$. Alternatively, it sends $M$ to
the fiber at $1\in \bb{A}^1$ of the Fourier transform of $M$. 

Let $X$ be an algebraic variety of dimension $n$. For a regular function $g\colon X\to \bb{A}^1$ and an integer $i$, we consider the following exponential mixed Hodge structures 
\begin{equation}\label{eq:EMHS-functions}
	\mathrm{H}^i(X,g):=\Pi(\mathcal{H}^{i-n}g_*\bb{Q}_X^{\rr{H}}),
    \quad
     \mathrm{H}^i_c(X,g):=\Pi(\mathcal{H}^{i-n}g_!\bb{Q}_X^{\rr{H}}).
\end{equation}
The exponential mixed Hodge structures $\rr{H}^i(X,g)$, $\rr{H}_c^i(X,g)$ are mixed of weights at least $i$ and mixed of weights at most $i$ respectively by \cite[A.19]{FresanEtAlHodgetheory22}. In particular, for a proper function $g$, we have $\rr{H}^i(X,g)=\rr{H}_c^i(X,g)$ pure of weight $i$.

\subsection{Irregular Hodge filtrations on the de Rham fibers}
On the de Rham fiber functor of objects in $\mathrm{EMHS}$, one can associate an \textit{irregular Hodge filtration} $F_{\rr{irr}}^\bullet \Xi_{\rr{dR}}M$, constructed using a
generalization of Deligne's filtration \cite[\S6]{SabbahFourierLaplace10}; see also   \cite[Prop. 2.61 \& Def. 3.2]{SabbahIrregularHodge18}.

Let  $\bb{A}^1_v$ be the dual affine line of $\bb{A}^1_t$ with coordinate $v$, the Fourier--Laplace transform $\mathrm{FL}$ of a $\bb{C}[t]\langle \partial_t\rangle$-module $N$, denoted by $\rr{FL}(N)$, has the same underlying set as that of $N$, which is viewed as a $\bb{C}[v]\langle \partial_v\rangle$-module with  the actions of $v$ and $\partial_v$ on $\rr{FL}(N)$ being defined by $-\partial_t$ and $t$ respectively. 

To an exponential mixed Hodge structure $M$, viewed as a mixed Hodge module on $\bb{A}^1_t$, we denoted by $(M,F^\bullet M)$ the underlying filtered $\scr{D}$-module of $M$. By \cite[Prop. 12.3.5]{KatzExponentialsums90}, we have 
$$
    \rr{FL}(M)=\rr{FL}\circ \Pi(M)=j_+j^+\rr{FL}(M),
$$
which implies that $G:=\rr{FL}(M)$ is localized, i.e., $G=\mathbb{C}[v,v\inv]\otimes_{\bb{C}[t]}G$.

The associated Brieskorn lattice $G^{0}=G_{(F)}^{0}$ is a free $\mathbb{C}[u]$-module $ (u=v^{-1})$ equipped with an action of $u^{2} \partial_{u}$, where   Formally, $G^0$ is defined as 
\[
G^{0}_{(F)}  :=  \sum_{p\in\mathbb Z} u^{p}\,\operatorname{FL}\bigl(F^{-p}M\bigr),
\]
(see e.g. \cite[Appx.]{SabbahYuirregularHodge15} for details), and $\{G_{(F)}^{p}=u^{p} G^{0}\}$ defines  a  decreasing filtration  on $G$. Let $V^\ast G$ be the $V$-filtration of $G$ with respect to the function $t$. The irregular Hodge filtration on the de Rham fiber $\Xi_{\rr{dR}}M=G/(u-1)G=G^0/(u-1)G^0$ is defined by the formula

\begin{equation}\label{eq:irregular-Hodge-filtration}
    F_{\mathrm{irr}}^{\gamma} \Xi_{\rr{dR}}M=\left(V^{\gamma} \cap G^{0}\right) /\left(V^{>\gamma} \cap G^{0}\right) \cap(u-1) G^{0} .
\end{equation}
 
For exponential mixed Hodge structures coming from geometry, the de Rham fiber of $\mathrm{H}^i_{?}(X,g)$ are isomorphic to $\mathrm{H}^i_{\mathrm{dR},?}(X,g)$ for $?\in \{\emptyset,c\}$. In this case,  Esnault, Sabbah, and Yu showed in \cite[\S1]{EsnaultEtAl$E_1$degenerationirregular17} that the irregular Hodge filtration on the de~Rham fiber coincides with the Kontsevich filtration and the Yu filtration \cite{YuIrregularHodge14} on the twisted de Rham cohomologies $\mathrm{H}^i_{\mathrm{dR},?}(X,g)$.

\section{Stationary phase formula}\label{section:stationary-phase} 
For a mixed Hodge module on the affine line with non-unipotent monodromy at $\infty$, Sabbah and Yu proved a stationary phase formula \cite[(7)]{SabbahYuIrregularHodge19}, relating its limiting Hodge numbers at infinity with the irregular Hodge numbers of the fiber of its Fourier transform. In this section, we will give a modified version of the stationary phase formula for exponential mixed Hodge modules, without the non-unipotency assumption.
\begin{prop}\label{prop:stationary-phase}
     Given an exponential mixed Hodge structure $M$, for any $\alpha\in [0,1)$ and $p\in \bb{Z}$, we have
        $$\dim \rr{gr}^{\alpha+p}_{F_{\rr{irr}}}\Xi_{\rr{dR}}(M) = \dim
         \rr{gr}^p_F \psi_{t\inv,\lambda}M  ,$$
    where $\lambda=\exp(-2\pi i \alpha)$.
\end{prop}
  This formula is crucial for the proof of \cref{theorem:main} and \cref{intro::non-degenerate}. 

\subsection{Proof of the stationary phase formula}
Let $M$ be an exponential mixed Hodge structure, which is a mixed Hodge module on $\mathbb{A}^1_t$ with vanishing cohomologies. By abuse of notation, we denote the underlying filtered $\scr{D}$-module of  $M$ by $\left(M, F^{\bullet} M\right)$, which is a filtered regular holonomic $\mathbb{C}[t]\left\langle\partial_{t}\right\rangle$-module. Let $\cc{M}$ be the $\scr{D}_{\bb{P}^1}$-module $j_+M$ where $j\colon \bb{A}^1_t\to \bb{P}^1$ the inclusion. Then $M=\Gamma(\bb{P}^1,\cc{M})=\Gamma(\bb{A}^1,\cc{M}|_{\bb{A}^1})$. 

We denote by $V^{\bullet} \mathcal{M}$ the $V$-filtration of $\mathcal{M}$ with respect to $t\inv$. For $\alpha \in \mathbb{R}$ and $\lambda=\exp (-2 \pi \mathrm{i} \alpha)$, we set $\psi_{t\inv, \lambda} \mathcal{M}=\operatorname{gr}_{V}^{\alpha} \mathcal{M}$. The Hodge filtration $F^{\bullet} M$ extends naturally to $V^{\alpha} \mathcal{M}$.
In such a way, $\left(\mathcal{M}, F^{\bullet} \mathcal{M}\right)$ is strictly specializable at $t\inv=0$. The space $\psi_{t\inv, \lambda} \mathcal{M}$ comes equipped with the induced filtration given by $F^{\bullet} \psi_{t\inv, \lambda} \mathcal{M}=F^{\bullet} \operatorname{gr}_{V}^{\alpha} \mathcal{M}$. 

 On the other hand, let $G$ be the (localized) Fourier--Laplace transform of $M$, and $V^{*} G$ be the $V$-filtration of $G$ with respect to the function $v$. We set similarly $\psi_{v, \lambda} G=\operatorname{gr}_{V}^{\alpha} G$, and the filtration $G_{(F)}^{\bullet}$ induces on it the filtration $G_{(F)}^{\bullet} \psi_{v, \lambda} G$.

\begin{lem}\label{lem:eq:(6)}
    Notation as above, we have 
    \[
        \dim \rr{gr}^p_{G_{(F)}}  \psi_{v, \lambda}  G =  \dim \rr{gr}^p_{F}  \psi_{t^{\prime}, \lambda} \cc{M}  
    \]
    for any $p$ and any $\lambda=\exp(-2\pi i \alpha)$ with $\alpha\in [0,1)$.
\end{lem}

\begin{proof}
    
For a pure Hodge module $M'$, let $\mathcal{M}'=j_*M'$ where $j\colon \bb{A}^1\hookrightarrow \mathbb{P}^1$. For the associated Rees module $R_F\mathcal{M}'$, we apply the ``inverse stationary phase formula'' of \cite[Prop.\,4.1(iv)]{SabbahMonodromyInfinity06} together with \cite[Lem. 5.20(*)$_\infty$]{SabbahFourierLaplace10}, we have 
    \begin{equation*} \begin{cases}
        R_{G_{(F)}} \psi_{v, \lambda} \rr{FL}(M') = R_F\psi_{t\inv, \lambda} \mathcal{M}' & \lambda \neq 1;\\
         R_{G_{(F)}}  \phi_{v, \lambda}  \rr{FL}(M') =  R_F \psi_{t\inv, \lambda} \mathcal{M}' &  \lambda = 1.\\
    \end{cases}
    \end{equation*} 
We take the associated graded pieces of the above identities and  deduce    
    \begin{equation}\label{eq:(6)}\begin{cases}
        \dim \rr{gr}^p_{G_{(F)}} \psi_{v, \lambda} \rr{FL}(M') = \dim \rr{gr}^p_F\psi_{t\inv, \lambda} \mathcal{M}' & \lambda \neq 1;\\
         \dim \rr{gr}^p_{G_{(F)}}  \phi_{v, \lambda}  \rr{FL}(M') =  \dim \rr{gr}^p_F \psi_{t\inv, \lambda} \mathcal{M}' &  \lambda = 1.\\
    \end{cases}
    \end{equation} 

    For our exponential mixed Hodge structure $M$, it is a mixed Hodge module on $\mathbb{A}^1$ with weight filtration $W^{\rr{MHM}}_\bullet$. For each $k$, we have an exact sequence
        $$0\to W^{\rr{MHM}}_{k-1}M \to W^{\rr{MHM}}_kM \to \rr{gr}_k^{W^{\rr{MHM}}}M \to 0.$$
    By induction on $k$ and applying \eqref{eq:(6)} to $\rr{gr}_k^WM$, we have 
        \begin{equation}
        \begin{cases}
        \dim \rr{gr}^p_{G_{(F)}} \psi_{v, \lambda} G =  \dim \rr{gr}^p_F\psi_{t\inv, \lambda} \mathcal{M}& \lambda \neq 1;\\
         \dim \rr{gr}^p_{G_{(F)}}  \phi_{v, \lambda}  G =  \dim \rr{gr}^p_F \psi_{t\inv, \lambda} \mathcal{M} &  \lambda = 1.\\
    \end{cases}
    \end{equation} 
    
    A priori, when $\lambda \neq 1$, the above formulas give exactly what we want. When $\lambda=1$, since $v$ is invertible on $G$, one has $\psi_{v,1}G=\phi_{v,1}G$ as well as $R_{G_{(F)}}\psi_{v,1}G=R_{G_{(F)}}\phi_{v,1}G$. Therefore, we have 
    $$
        \dim \rr{gr}^p_{G_{(F)}}  \psi_{v, 1}  G
         = \dim \rr{gr}^p_{G_{(F)}}  \phi_{v, 1}  G 
         = \dim \rr{gr}^p_F \psi_{t\inv, 1} \mathcal{M}
    $$
    for any $p$.
\end{proof}

Now we prove \cref{prop:stationary-phase}. 
\begin{proof}
We follow the argument in \cite[\S5]{SabbahYuIrregularHodge19}. By the formula of irregular Hodge numbers \eqref{eq:irregular-Hodge-filtration} and
\cite[(1.3)]{SabbahFourierLaplace10} (by replacing there $z$ with $u=v\inv$ and taking $z_{0}=1$ and $0$), we conclude that 
$$
\dim \rr{gr}_{F_{\text {irr }}}^{\gamma} \Xi_{\rr{dR}}(M)=\dim \left(\left(V^{\gamma} \cap G^{0}\right) /\left[\left(V^{>\gamma} \cap G^{0}\right)+\left(V^{\gamma} \cap G^{1}\right)\right]\right) .
$$

Let us set $\alpha=\{\gamma\}:=\gamma-[\gamma] \in[0,1)$, so that $v^{-[\gamma]}\left(V^{\gamma} \cap G^{0}\right)=V^{\alpha} \cap G^{[\gamma]}$. Then, for $p \in \mathbb{Z}$ and $\alpha \in[0,1)$, we find

\begin{equation}
    \dim \operatorname{gr}_{F_{\text {irr }}}^{\alpha+p} \Xi_{\rr{dR}}(M)
    =\dim \left(\left(V^{\alpha} \cap G^{p}\right) /\left[\left(V^{>\alpha} \cap G^{p}\right)+\left(V^{\alpha} \cap G^{1+p}\right)\right]\right) 
    =\dim \operatorname{gr}_{G_{(F)}}^{p} \psi_{v, \lambda} G .
\end{equation}
Together with \cref{lem:eq:(6)}, we obtain:
\begin{equation}\label{eq:stationary-phase-modified}
    \dim \rr{gr}^{\alpha+p}_{F_{\rr{irr}}}\Xi_{\rr{dR}}(M)=\dim \rr{gr}^p_F\psi_{t\inv,\lambda}M
\end{equation}
for $p \in \mathbb{Z}$ and $\alpha \in[0,1)$.
\end{proof}

Applying \cref{prop:stationary-phase} to the exponential mixed Hodge structure $\mathrm{H}^i(U,f)$ defined in \eqref{eq:EMHS-functions}, we deduce the following corollary.

\begin{cor}\label{cor:stationary-phase-functions}
    Let $f$ be a regular function $f$ on a smooth quasi-projective variety $U$. For and any integers $i$ and $p$, and any $\alpha\in [0,1)$ we have
        $$\dim \rr{gr}^{\alpha+p}_{F_{\rr{irr}}}\mathrm{H}^i_{\mathrm{dR}}(U,f)= \dim
         \rr{gr}^p_F\psi_{t\inv,\lambda}\Pi\,\mathrm{R}^i f_*\mathbb{Q}^{\rr{H}},$$
    where $\lambda=\exp(-2\pi i \alpha)$.
\end{cor}

\section{Classical, irregular, and Landau--Ginzburg Hodge numbers}\label{section:KKP-numbers}

\subsection{Classical and irregular Hodge numbers}\label{subsection:classical-numbers}
We first give a proof of \cref{theorem:main} using \cref{cor:stationary-phase-functions}.
\begin{prop}\label{prop:distinguished-triangle}
    There is a distinguished triangle 
    $$
         f_*\bb{Q}_{ U}^{\rr{H}} \to \Pi\,    f_*\bb{Q}_{ U}^{\rr{H}} \to \rr{R}\Gamma( U,\bb{Q}^{\rr{H}}_{  U})\otimes_{\bb{Q}}\bb{Q}_{\bb{A}^1}^{\rr{H}}\xrightarrow{+1}
    $$ 
    in $\rr{D}^b(\rr{MHM}(\bb{A}^1))$, where $\Pi$ is the functor defined in \eqref{eq:Pi}.
\end{prop}
\begin{proof}
 Consider the following diagram
    $$\begin{tikzcd}
          U \times \bb{G}_m \ar[r, hook, "j"]\ar[d,"  f\times \rr{id}"] 
        &   U\times \bb{A}^1 \ar[r,"\rr{pr}_1"] \ar[d,"  f\times \rr{id}"]
        &   U \ar[d, "  f"] \\
        \bb{A}^1 \times \bb{G}_m \ar[r,hook, "j"]
        & \bb{A}^1 \times \bb{A}^1 \ar[r,"\rr{pr}_1"]
        & \bb{A}^1
    \end{tikzcd}$$
    where $\rr{pr}_1$ is the projection to the first factor.
    By the definition of $\Pi$, we have
    \begin{equation}\label{eq:open-part}
        \begin{split}
            \Pi\,    f_*\bb{Q}_{  U}^{\rr{H}}
        &= \rr{sum}_*(  f_*\bb{Q}_{ U}^{\rr{H}} \boxtimes j_!\bb{Q}^{\rr{H}}_{  \bb{G}_{m}}) \\
        &= \rr{sum}_*(\pr_1^* f_*\bb{Q}_{ U}^{\rr{H}} \otimes j_!\bb{Q}^{\rr{H}}_{\bb{A}^1\times \bb{G}_{m}}) \\
        &= \rr{sum}_*( (f \times \rr{id})_*\bb{Q}_{ U\times \bb{G}_m}^{\rr{H}} \otimes j_!\bb{Q}^{\rr{H}}_{\bb{A}^1\times \bb{G}_{m}}) 
        \end{split}
    \end{equation}
    where we applied the smooth base change formula in the last isomorphism.

    To the triple $(\bb{A}^1\times \bb{A}^1, \bb{A}^1\times \bb{G}_m, \bb{A}^1\times 0)$,
    we have the distinguished triangle
    \begin{equation}\label{eq:triangle-j-i}
        j_!\bb{Q}^{\rr{H}}_{ \bb{A}^1\times \bb{G}_m} \to \bb{Q}^{\rr{H}}_{ \bb{A}^1\times \bb{A}^1} \to i_*i^*\bb{Q}^{\rr{H}}_{ \bb{A}^1\times \bb{A}^1} \xrightarrow{+1}.
    \end{equation}
    Notice that 
    \begin{equation}\label{eq:whole-part}
        \begin{split}
        &\rr{sum}_*((  f \times \rr{id})_*\bb{Q}_{  U\times \bb{G}_m}^{\rr{H}}\otimes  \bb{Q}^{\rr{H}}_{ \bb{A}^1\times \bb{A}^1})\\
        =&\rr{sum}_*(  f \times \rr{id})_*\bb{Q}_{  U\times \bb{A}^1}^{\rr{H}} \\
        = &\pr_{2*} \iota_*\bb{Q}_{  U\times \bb{A}^1}^{\rr{H}}  
        \simeq \pr_{2*}  \bb{Q}_{ U\times \bb{A}^1}^{\rr{H}} 
        \end{split}
    \end{equation}
    where $\iota\colon   U \times \bb{A}^1\to  U\times \bb{A}^1$ the isomorphism defined by $(x,y)\mapsto (x,y+f(x))$ and $\pr_2$ is the projection to the second factor, as illustrated in the following commutative diagram:
    $$\begin{tikzcd}
          U \times \bb{A}^1_y \ar[d, "\iota","\sim"'{anchor=south, rotate=90}]\ar[r," f\times \rr{id}"] 
        & \bb{A}^1_x \times \bb{A}^1_y  \ar[d, "{(x,y+x)}","\sim"'{anchor=south, rotate=90}]\ar[r,"\rr{sum}"]
        & \bb{A}^1 \ar[d, equal] \\
         U \times \bb{A}^1 \ar[rr,bend right,"\rr{pr}_2"]\ar[r,"  f\times \rr{id}"]
        & \bb{A}^1 \times \bb{A}^1 \ar[r,"\rr{pr}_2"]
        & \bb{A}^1.
    \end{tikzcd}$$

    Similarly, we have
    \begin{equation}\label{eq:closed-part}
        \begin{split}
        &\rr{sum}_*((  f \times \rr{id})_*\bb{Q}_{  U\times \bb{A}^1}^{\rr{H}}\otimes  i_*i^*\bb{Q}^{\rr{H}}_{ \bb{A}^1\times  \bb{A}^1})\\
        =&(\rr{sum}_* i_*)(i^*(  f \times \rr{id})_*\bb{Q}_{ U\times  \bb{A}^1}^{\rr{H}}\otimes   i^*\bb{Q}^{\rr{H}}_{ \bb{A}^1\times  \bb{A}^1}) \\ 
        =&\rr{id}_*i^*((  f \times \rr{id})_*\bb{Q}_{ U\times  \bb{A}^1}^{\rr{H}}\otimes    \bb{Q}^{\rr{H}}_{ \bb{A}^1\times  \bb{A}^1}) \\ 
        =& i^* ( (f \times \rr{id})_*\rr{pr}_{1}^*\bb{Q}_{ U}^{\rr{H}} [1] )\\ 
        =& ( i^* \rr{pr}_1^*)f_*\bb{Q}_{ U}^{\rr{H}}[1]
        =f_*\bb{Q}_{  U}^{\rr{H}}[1],\\ 
        \end{split}
    \end{equation}
    where we used the projection formula in the first isomorphism, the smooth base change formula in the forth isomorphism, and denoted by $i$ the inclusion morphisms below
    $$\begin{tikzcd}
         {U}=  {U}\times 0 \ar[d,"  f"] \ar[r,"i"]&   U\times \bb{A}^1 \ar[d,"  f\times \rr{id}"] & \\
        \bb{A}^1=\bb{A}^1\times 0 \ar[r,"i"] \ar[rd,"\rr{id}"]  & \bb{A}^1\times \bb{A}^1 \ar[d,"\rr{sum}"]&  \\
         0 \ar[r,"i"]&\bb{A}^1.&
    \end{tikzcd}$$
    Combining \eqref{eq:open-part}, \eqref{eq:whole-part}, \eqref{eq:closed-part} and applying $\rr{sum}_*( (  f\times \rr{id})_*\bb{Q}_{  U\times \bb{A}^1}^{\rr{H}} \otimes -)$ to the distinguished triangle \eqref{eq:triangle-j-i}, we have 
    $$
        \Pi   f_*\bb{Q}_{  U}^{\rr{H}} 
        \to  \pr_{2*}  \bb{Q}_{  U\times \bb{A}^1}^{\rr{H}}
        \to   f_*\bb{Q}_{  U}^{\rr{H}}[1] \xrightarrow{+1}.
    $$
    To conclude the proof, it suffices to notice that 
    $$
        \pr_{2*}  \bb{Q}_{  U\times \bb{A}^1}^{\rr{H}}
        \simeq  \rr{R}\Gamma(  U,\bb{Q}^{\rr{H}}_{  U})\otimes_{\bb{Q}}\bb{Q}_{\bb{A}^1}^{\rr{H}}
    $$
    and shift the triangle.
\end{proof}

\begin{thm}\label{thm:KKP-numbers}
    For  any $\alpha \in [0,1) $ and
    \(p \in \mathbb{Z}\), we have
    \[
        h^{p+\alpha}_{\rr{irr}}(U,f,i) = \dim\rr{gr}_{F_{\rr{lim}}}^p\mathrm{H}^{i}(U^{\mathrm{an}}, f^{-1}(t)^{\mathrm{an}})_\lambda,
    \]
    where $\lambda=\exp(-2\pi\mathrm{i}\alpha)$. 
\end{thm}
\begin{proof}
Taking the nearby cycle functor $\psi_{1/t}$ to the distinguished triangle in  \cref{prop:distinguished-triangle}, we have a distinguished triangle 
\begin{equation}\label{eq:distinguished-triangle-limiting}
    \psi_{1/t} f_*\bb{Q}_{ U}^{\rr{H}} \to \psi_{1/t}\Pi\,   f_*\bb{Q}_{ U}^{\rr{H}} \to \rr{R}\Gamma(U,\bb{Q}^{\rr{H}}_{ U}) \xrightarrow{+1}.
\end{equation}

Notice that   
$
    \psi_{1/t,\lambda}\cc{H}^{i-n} f_*\bb{Q}_{ U}^{\rr{H}} 
$
is supported at $\infty$ and is identified with
$
\rr{H}^{i-1}(  f^{-1}(t)^{\mathrm{an}})_\lambda
$
equipped with the limiting Hodge structure for generic $t\in \bb{C}$. Therefore, by the definition of relative cohomology, we deduce from \eqref{eq:distinguished-triangle-limiting} that
\begin{equation}\label{eq:distinguished-triangle-limiting3}
    \gr_F^p \psi_{1/t,\lambda}\Pi\,\rr{R}^{i-n}f_*\bb{Q}_{ U}^{\rr{H}} 
    =\gr_{F_{\rr{lim}}}^p \rr{H}^i(U^{\mathrm{an}}, f^{-1}(t)^{\mathrm{an}})_\lambda.
\end{equation}

On the other hand, by \cref{cor:stationary-phase-functions}, we have
    $$\dim \rr{gr}^{p+\alpha}_{F_{\rr{irr}}}\mathrm{H}^i_{\mathrm{dR}}(U,f)= \dim
    \rr{gr}^p_F\psi_{t\inv,\lambda }\Pi\,\mathrm{R}^{i-n}f_*\mathbb{Q}^{\rr{H}}$$
where $\lambda=\exp(-2\pi i\alpha)$, which finishes the proof. 
\end{proof}
\subsection{A conjecture of Katzarkov--Kontsevich--Pantev}\label{subsection:KKP-conjecture}
In \cite{KatzarkovEtAlBogomolovTian17}, Katzarkov, Kontsevich, and Pantev introduced certain numerical invariants  $f_{\rr{LG}}^{p,q}$ and $h_{\rr{LG}}^{p,q}$  
for Landau--Ginzburg models. We recall the definitions as follows.

We first recall the definition of the number $f^{p,q}_{\rr{LG}}$. Starting from a regular function $f$ on a smooth quasi-projective variety $U$, we can always find a \textit{good compactification} $(X,g)$, consisting of a smooth projective variety $X$ with  $D=X-U$ being simple normal crossing, and a morphism $g\colon X\to \bb{P}^1$ extending $f$. 

Let  $P=\sum_{i=1}^r e_i D_i$ be the pole divisor of $f$ and for a rational number $\alpha\in [0,1)$, we denote by $\lfloor\alpha P\rfloor$ the divisor $\sum_{i=1}^r \lfloor\alpha e_i\rfloor D_i$. The \textit{Kontsevich complex} $\Omega^{\bullet}_{X}(\log D ,f)(\lfloor\alpha P\rfloor)$ is defined as 
    $$
        \Omega^{\bullet}_{X}(\log D ,f)(\lfloor\alpha P\rfloor) =  \{ \omega \in \Omega^{\bullet}_{X}(\log D)(\lfloor\mu P\rfloor) \mid df\wedge \omega \in \Omega^{\bullet}_{X}(\log D)(\lfloor(\alpha+1) P\rfloor)  \}.
    $$
\begin{defn}\label{defn:f^{p,q}}
    For each $p\in \bb{Z}_{\geq 0}$ and $\alpha\in [0,1)$ we define
        $$f^{p,q-\alpha}_{\rr{LG}}(U,f):=\dim    \mathbb{H}^{p} (X,\Omega^{q}_{X}(\log D,f)(\lfloor\alpha P\rfloor)).$$
\end{defn}

\begin{rmk}\label{rmk:irregular-Hodge-numbers}
    In \cite[Cor. 1.4.5]{EsnaultEtAl$E_1$degenerationirregular17}, the Landau--Ginzburg Hodge numbers $f^{p,q-\alpha}_{\rr{LG}}(U,f)$ are identified with the irregular Hodge numbers $h^{q-\alpha}_{\rr{irr}}(U,f,p+q)$, which are independent of the choice of the good compactification $(X,g)$ of $(U,f)$.
\end{rmk}

Now we recall the definition of the number  $h^{p,q}_{\rr{LG}}$.
Let $t\in \mathbb{A}^1$ be a regular value of $f$ near infinity.
Consider the monodromy transformation on the relative singular cohomology
\[
T \colon \rr{H}^{i}(U^{\mathrm{an}}, f^{-1}(t)^{\mathrm{an}})
\longrightarrow \rr{H}^{i}(U^{\mathrm{an}}, f^{-1}(t)^{\mathrm{an}}),
\]
which corresponds to moving the smooth fiber $f\inv(t)$ once around infinity. The monodromy operator $N$ induces   the monodromy-weight filtration $W(N,\bullet)$ on $\rr{H}^{i}(U^{\mathrm{an}}, f^{-1}(t)^{\mathrm{an}})$.

\begin{defn}
    The \emph{Landau--Ginzburg Hodge numbers} $h_{\rr{LG}}^{p,q}(U,f)$ are defined by
    $$
        h^{p,q}_{\rr{LG}}(U,f):= \dim_{\mathbb{C}}\gr^{W(N,p+q)}_{2p}\rr{H}^{p+q}(U, U_{t}; \mathbb{C}).
    $$
\end{defn}
 
Katzarkov, Kontsevich, and Pantev expected these two types of Landau--Ginzburg Hodge numbers to coincide.

\begin{conj}[{\cite[Conj. 3.6]{KatzarkovEtAlBogomolovTian17}}]\label{conj:KKP}
    Let $(U,f)$ be a Landau--Ginzburg model, where $U$ is a smooth quasi-projective variety and $f$ a regular function $f\colon U\to \bb{A}^1$ with reduced pole divisor. Then for all $p,q\in \bb{N}$, we have
        $$f^{p,q}_{\rr{LG}}(U,f)=h^{p,q}_{\rr{LG}}(U,f).$$
\end{conj}

\begin{prop}\label{prop:kkp}
    The above conjecture holds if and only if the limiting Hodge structures on $\rr{H}^{\bullet}(U^{\mathrm{an}}, f^{-1}(t)^{\mathrm{an}})$ is of Hodge--Tate type.
\end{prop}
\begin{proof}
    Notice that when $f$ has reduced pole divisor, the eigenspaces $$\rr{H}^{p+q}(U^{\mathrm{an}}, f^{-1}(t)^{\mathrm{an}})_\lambda=0$$ 
    unless $\lambda=1$. By \cref{rmk:irregular-Hodge-numbers}, the numbers $f^{p,q}_{\rr{LG}}(U,f)$ are identified with the irregular Hodge numbers $h^{q}_{\rr{irr}}(U,f,p+q)$. On the other hand, by \cref{thm:KKP-numbers}, the numbers $h^{q}_{\rr{irr}}(U,f,p+q)$ are identified with the limiting Hodge numbers
    $$ \dim \rr{gr}_{F_{\rr{lim}}}^q \rr{H}^{p+q}(U^{\mathrm{an}}, f^{-1}(t)^{\mathrm{an}})=\dim \rr{gr}_{F_{\rr{lim}}}^q \rr{H}^{p+q}(U^{\mathrm{an}}, f^{-1}(t)^{\mathrm{an}})_1,$$
    which coincides with $h^{p,q}_{\rr{LG}}(U,f)=h^{q,p}_{\rr{LG}}(U,f)$ if and only if the limiting Hodge structure on $\rr{H}^{p+q}(U^{\mathrm{an}}, f^{-1}(t)^{\mathrm{an}})$ is of Hodge--Tate type. 
\end{proof}

\section{Non-degenerate functions} \label{section:snc-pairs}

\subsection{Non-degeneracy conditions}\label{subsection:non-degenerate-functions}
By a \textit{simple normal crossing pair}, we mean a pair $(X,D)$ consisting of a smooth projective variety $X$ and a simple normal crossing divisor $D$ in $X$. In this section, we introduce non-degenerate functions on simple normal crossing pairs as defined by Katz and Mochizuki, and prove basic purity and vanishing statements for their associated exponential mixed Hodge structures.

There is a notion of non-degeneracy in \cite{MochizukiTwistorGKZ15} in the analytic setting as follows:
\begin{defn}\label{defn:mochizuki-non-degenerate}
    Let \( X \) be a complex manifold with a simple normal crossing hypersurface \( D \). Let \( f \) be a meromorphic function on \( (X,D) \).  The function \( f \) is called non-degenerate along \( D \) if for a small neighborhood \( N \) of \( |(f)_\infty| \), 
\begin{itemize}
    \item \( (f)_0 \cap N \) is reduced and non-singular, and
    \item \( N \cap (|(f)_0| \cup D) \) is normal crossing.
\end{itemize} 
\end{defn}
 
In the algebraic setting, we will use the following definition of non-degenerate functions by Katz \cite[\S4]{KatzSommesexponentielles80}.

\begin{defn}\label{defn:non-degenerate}
    A regular function $f$ on a smooth quasi-projective variety $U$ is called \textit{\mbox{non-degenerate}} if
    \begin{enumerate}
        \item there exists a smooth projective variety $X$ and a simple normal crossing divisor $D=\sum_{i=1}^r e_i D_i$ for some positive integers $e_i$ such that $U=X-D$;
        \item there exist a global section $s_0\in \Gamma(X,\cc{O}_X(D))$ and global sections $\sigma_i\in \Gamma(X,\cc{O}_X(D_i))$ such that $D=Z(\sigma_1^{e_1}\dots \sigma_r^{e_r})$ and $f$ agrees with the function
            $$\frac{s_0}{s_\infty}\colon X- D\to \bb{A}^1,$$
        where $s_\infty=\sigma_1^{e_1}\dots \sigma_r^{e_r}$. 
        \item $Z(s_0)$ is smooth and  transversal to $D$.
    \end{enumerate}
    Moreover, a non-degenerate function is called \textit{strongly non-degenerate} if the line bundles $\scr{L}_i$ corresponding to $D_i$ are all ample.
\end{defn}

If we take a non-degenerate function in the sense of \cref{defn:non-degenerate}, then its analytification is non-degenerate in the sense of \cref{defn:mochizuki-non-degenerate}. From now on, we will simply say that $f$ is a non-degenerate function on the simple normal crossing pair $(X,D)$.

\begin{exe}\label{rmk:compactified-LG-comparision}
In \cite{KatzarkovEtAlBogomolovTian17}, the authors introduced the notion of \textit{compactified Landau--Ginzburg models} $((Z,f),D_Z)$, consisting of a smooth projective variety $Z$, a flat morphism
$f \colon Z \to \mathbb{P}^1$, and a reduced simple normal crossings divisor
\[
D_Z = D^h \cup D^v,
\]
where $D^h$ is the horizontal divisor and $D^v$ the pole divisor of $f$, subject to additional conditions.
 
If the horizontal divisor $D^h$ is empty, then $D_Z$ is exactly the pole divisor of $f$, and the zero locus of $f$ is disjoint from the pole divisor $D_Z$. Hence, the compactified Landau--Ginzburg model $((Z,f),D_Z)$ defines a non-degenerate function in the sense of \cref{defn:non-degenerate} on the simple normal crossing pair $(Z,D_Z)$.

Compared with the non-degeneracy condition in \cref{defn:non-degenerate}, we don't require the divisor to be reduced. However, we impose the condition that $D_Z$ is precisely the pole divisor of $f$, rather than allowing an additional horizontal component $D^h$, which is motivated by the purity results that will be established in \cref{lemma:snc-pair-implies-pure}.
\end{exe}
 
\begin{exe}
    Given a Laurent polynomial $f=\sum_{\gamma\in \bb{Z}^d} c_\gamma x^\gamma \in \bb{C}[x_1^{\pm},\dots,x_n^{\pm}]$ on the algebraic torus $(\bb{C}^*)^n$, it is called \textit{convenient} if $0$ is an interior point of its Newton polytope $\Delta$. It is called \textit{non-degenerate} with respect to its Newton polytope $\Delta$ if for any face $\delta$ of $\Delta$ (including $\Delta$ itself), the system of equations
        $$f_\delta=\frac{\partial f_\delta}{\partial x_1}=\cdots=\frac{\partial f_\delta}{\partial x_n}=0$$
    has no solutions in $(\bb{C}^*)^n$, where $f_\delta=\sum_{\gamma \in \delta} c_\gamma x^\gamma $ is the restriction of $f$ to the face $\delta$. 

    Actually, a non-degenerate Laurent polynomial is non-degenerate in the sense of \cref{defn:non-degenerate}. In fact, after choosing a smooth proper toric variety $X(f)$ associated with a refinement of the normal fan of the Newton polytope $\Delta$, the boundary divisor $D^v\cup D^h =X-(\mathbb{C}\cros)^N$ is simple normal crossing and $f$ can be extended to a regular function on $U:=X- D^v$. By \cite[Prop. 4.3]{YuIrregularHodge14}, the zero locus $Z(f)$ is transversal to the vertical divisor $D^v$. Hence, $f$ is non-degenerate on the simple normal crossing pair $(X(f),D^v)$. In particular, if $f$ is convenient and non-degenerate, then the boundary divisor is $D^v$ and the interior $U$ is exactly $(\bb{C}^*)^n$.
\end{exe}

\subsection{Purity and vanishing}\label{subsection:vanishing}
Given a non-degenerate function $f$ on a simple normal crossing pair $(X,D)$, then we have two associated exponential mixed Hodge structures $\mathrm{H}^i(U,f)$ and $\mathrm{H}^i_c(U,f)$ defined in \eqref{eq:EMHS-functions}. The following proposition shows that they are isomorphic and pure of weight $i$.

\begin{prop}\label{lemma:snc-pair-implies-pure}
    Given a non-degenerate function $f$ on a simple normal crossing pair $(X,D)$, then  the associated exponential mixed Hodge structures $\mathrm{H}^i(U,f)$ and $\mathrm{H}^i_c(U,f)$ are isomorphic. Moreover, they are pure of weight $i$ for all $i$.        
\end{prop}

\begin{proof}

    By the construction of $f$, there exist $s_\infty\in \rr{H}^0(X,\cc{O}_X(D))$ and $s_0\in \rr{H}^0(X,\cc{O}_X(D))$  as in \cref{defn:non-degenerate} such that
        $$f :=s_0/s_\infty\colon U\to \bb{A}^1.$$
    We denote by $B=Z(s_0)\cap D$ the base locus.     Consider
        $$\mathfrak{X}=\{(t,x)\in \bb{P}^1\times X\mid ts_{\infty}-s_0=0\}$$
    and let $g\colon \fk{X}\to \bb{P}^1$ be the projection to the first factor. One can easily check that $\tilde U:=g\inv(\bb{A}^1)=(\bb{A}^1\times B)\cup U$ is smooth using the Jacobian criterion. Hence, we have the following commutative diagram
\begin{equation}\label{eq:snc-pair-diagram}
\begin{tikzcd}
    \bb{A}^1\times B \ar[d,"\pr_1"] \ar[r,"i"] & \tilde{U}=g\inv(\bb{A}^1)\ar[d,"\tilde f"]   & U\ar[d,"f"]\ar[l,hook',"j"']\\
    \bb{A}^1 &\bb{A}^1\ar[l,equal] &\bb{A}^1\ar[l,equal]
\end{tikzcd}
\end{equation}
where $\tilde f=g|_{\tilde U}$. For each $J\subset \{1,\dots,r\}$, we set $D_J=\cap_{j\in J}D_j$ and $B_{J}=B \cap D_J$. Notice that $\bb{A}^1\times B_j$ indeed form a simple normal crossing divisor in $g\inv(\bb{A}^1)$.

First consider the distinguished triangle
    $$
        f_!\bb{Q}_U^{\rr{H}}\to \tilde{f}_*\bb{Q}_{\tilde{U}}^{\rr{H}}\to \rr{pr}_{1*}i_*i^*\bb{Q}_{\tilde{U}}^{\rr{H}}\xrightarrow{+1}.
    $$
Notice that $\rr{pr}_{1*}i_*i^*\bb{Q}_{\tilde{U}}^{\rr{H}}=\rr{pr}_{1*}\bb{Q}_{\bb{A}^1\times B}^{\rr{H}}[1]$  is a constant mixed Hodge module on $\bb{A}^1$. So the functor $\Pi$ annihilates the cone of $f_!\bb{Q}_U^{\rr{H}}\to \tilde{f}_*\bb{Q}_U^{\rr{H}}$ and yields the  isomorphism $\mathrm{H}^i_c(U,f)\to\mathrm{H}^i (\tilde{U},\tilde{f}) $ for any $i$.
 
Similarly, consider the dual distinguished triangle
$$
    \rr{pr}_{1*}i^!\bb{Q}^{\rr{H}}_{\tilde{U}} \to \tilde{f}_*\bb{Q}^{\rr{H}}_{\tilde{U}}\to f_*\bb{Q}^{\rr{H}}_{U} \xrightarrow{+1}.
$$
Let $B^{(s)}=\coprod_{|J|=s} B^J$, which is equipped with a natural morphism $ i_s\colon \bb{A}^1\times B^{(s)}\to \bb{A}^1\times B $. Using the fact that $(i\circ i_s)^!\bb{Q}_X^{\rr{H}}=\bb{Q}_{\bb{A}^1\times B^{(s)}}^{\rr{H}}[-2s](-s)$, we have the following resolution of $ i^!\bb{Q}^{\rr{H}}_{\tilde U}$:
$$
    0\leftarrow i^!\bb{Q}^{\rr{H}}_{\tilde U} \leftarrow i_{1*}\bb{Q}^{\rr{H}}_{\bb{A}^1\times B^{(1)}}[-2](-1)\leftarrow i_{2*}\bb{Q}^{\rr{H}}_{\bb{A}^1\times B^{(2)}}[-4](-2) \leftarrow \cdots,
$$
where the morphisms are induced by the iterating sums of the inclusion morphisms. Hence, $\rr{pr}_{1*}i^!\bb{Q}^{\rr{H}}_{\tilde{U}}$ is isomorphic to the iterated cone of
\[
    [\rr{pr}_{1*}\bb{Q}^{\rr{H}}_{\bb{A}^1\times B^{(1)}}[-2](-1)\leftarrow \rr{pr}_{1*}\bb{Q}^{\rr{H}}_{\bb{A}^1\times B^{(2)}}[-4](-2)\leftarrow \cdots ].
\]
Similar to the above, the functor $\Pi$ annihilates these constant mixed Hodge modules $\rr{pr}_{1*}\bb{Q}^{\rr{H}}_{\bb{A}^1\times B^{(d)}}[-2d](-d)$ on $\bb{A}^1$. So it also annihilates the cone of $\tilde{f}_*\bb{Q}_{\tilde{U}}^{\rr{H}}\to f_*\bb{Q}_U^{\rr{H}}$ and yields the  isomorphism $\mathrm{H}^i(\tilde{U},\tilde{f})\to\mathrm{H}^i (U,f) $ for any $i$.

Combining the above discussion, we deduce an isomorphism 
        $$\mathrm{H}^i_c(U,f)\to \mathrm{H}^i(\tilde{U},\tilde{f})\to \mathrm{H}^i(U,f)$$
on the exponential mixed Hodge structures for any $i$. Then, the  pureness follows from the fact that $\mathrm{H}^i_{\rr{dR},c}(U,f)$ and $\mathrm{H}^i_{\rr{dR}}(U,f)$ are mixed of weights $\leq i$ and $\geq i$ respectively.
\end{proof}

\begin{cor}\label{cor:vanishing}
    If $f$ is strongly non-degenerate, the exponential mixed Hodge structures $\mathrm{H}^i(U,f)$ are zero unless $i=\dim U$. 
\end{cor}
\begin{proof}
    Because $U$ is affine, the twisted de Rham cohomology $\mathrm{H}^i_{\rr{dR}}(U,f)$ is zero unless $i\leq \dim U$, and the twisted de Rham cohomology with compact support $\mathrm{H}^i_{\rr{dR},c}(U,f)$ is zero unless $i\geq \dim U$. By \cref{lemma:snc-pair-implies-pure}, they are isomorphic. Therefore, they are both zero unless $i=\dim U$.
\end{proof}

\begin{exe}\label{example}
% \subsubsection{}
Let $X=\bb{P}^2$ and $D=L_0\cup L_1\cup L_2$, where the three lines $L_0, L_1$ and $L_2$ are defined by $x_0,x_1$, and $x_2$ respectively. In particular, the interior $U:=X-|D|$ is the torus $\bb{G}_{m}^2$. The (strongly) non-degenerate functions $f$ on $(X,D)$ are of the form
     $$\frac{F}{x_0x_1x_2}\colon U\to \bb{A}^1$$
 where $F$ are cubic homogeneous polynomials in $x_0,x_1$, and $x_2$, such that $Z(F)$ is transverse to $D$. The corresponding irregular Hodge numbers will be calculated in \cref{example:calculation}.
\end{exe}

\begin{exe}\label{example2}
% \subsubsection{}
Let $X=\bb{P}^2$ and $D=Z(x_0^3+x_1^3+x_2^3+3x_0x_1x_2)$. We consider the regular function 
     $$f=\frac{x_0x_1x_2}{x_0^3+x_1^3+x_2^3+3x_0x_1x_2}\colon U\to \bb{A}^1$$
on the interior $U:=X-D$ as in \cref{example:counterexample}. It is degenerate because the zero locus $Z(x_0x_1x_2)$ is not smooth. However, the functions $f+c$ for general $c\in \bb{C}$ are (strongly) non-degenerate. In fact, the zero locus $Z(x_0x_1x_2+c(x_0^3+x_1^3+x_2^3+3x_0x_1x_2))$ is smooth and transverse to $D$ for general $c$ by Bertini's theorem. 

However, the irregular Hodge numbers of $(U,f)$ are the same as those of $(U,f+c)$ by \cref{theorem:main}. We will  calculate the irregular Hodge numbers in \cref{example:calculation2}.
\end{exe}

\subsection{Irregular Hodge numbers of non-degenerate functions}\label{sec:proof-of-main}
 
In this section, we prove \cref{intro::non-degenerate}. We keep the notations as in \cref{section:snc-pairs}.

Let $D=P=\sum_{i=1}^r e_iD_i$ be the simple normal crossing divisor in $X$ and  $s_\infty\in \rr{H}^0(X,\cc{O}_X(D))$ a fixed section as in \cref{defn:non-degenerate}. We denote by $S\subset \rr{H}^0(X,\cc{O}_X(D))$ the subset containing sections $s $ such that $B_s=Z(s)$ intersects $P$ transversally. For $s\in S$, we denote by
    $$f_s:=s/s_\infty\colon U\to \bb{A}^1$$
the non-degenerate function for $(X,D)$ associated with $s$ and by $B_s=Z(s)\cap P$ the base locus. We define
    $$\mathfrak{X}_s=\{(t,x)\in \bb{P}^1\times X\mid ts_{\infty}-s=0\}$$
and $g_s\colon \fk{X}_s \to \bb{P}^1$ be the projection to the first factor.

Notice that $U$ is embedded in $\fk{X}_s$ via the graph embedding of $f_s$. We define  $\tilde{U}_s=g_s\inv(\bb{A}^1)$, which  contains $U$.  We denote by the restrictions of $g_s$ to   $\tilde{U}_s$ by   $\tilde{f}_s$, where $\tilde{f}_s$ is proper. Then we have the following commutative diagram 
\begin{equation}\label{eq:diagram-KKP}
\begin{tikzcd}
    \fk{X}_s \ar[d,"g_s"]  & \tilde{U}_s =\tilde{g}_s\inv(\bb{A}^1)\ar[d,"\tilde{f}_s"] \ar[l,hook'] & U\ar[d,"f_s"]\ar[l,hook']\\
    \bb{P}^1 &\bb{A}^1\ar[l,hook'] &\bb{A}^1\ar[l,equal].
\end{tikzcd}
\end{equation}
 
\begin{lem}\label{lem:independence-2}
    For any $\alpha\in \bb{Q}$ and  $i\in \bb{N}$, the Hodge numbers of  $\psi_{1/t,\exp(-2\pi i \alpha)} \rr{R}^ig_{s*}\bb{Q}_{\cc{X}_s}^{\rr{H}}$  are independent of $s\in S$.
\end{lem}
\begin{proof}
    Using GAGA, it suffices to prove the same properties in the analytical setting. As $\rr{R}^ng_{s*}\bb{Q}\Hodge$ is a mixed Hodge module on $\bb{P}^1$, it restricts to a variation of Hodge structure outside critical values of $g_{s}$. Hence, there exists a real number $r(s)$ such that $\rr{R}^ng_{s*}\bb{Q}\Hodge$ restricts to a VHS on $\Delta_{>r(s)}=\{r\mid |r|>r(s)\}$.

    For any two points $s_1$ and $s_2$ in $S$, choose a compact disk $\bar{\Delta}\subset S$, such that $s_1$ and $s_2$ lies in the interior $\bar{\Delta}^\circ$. By the compactness of $\bar{\Delta}$, we find a uniform bound $r$ for $\{r(s)\mid s\in \bar\Delta\}$. In particular, the critical values of $g_s$ are contained in $\Delta_{\leq r}$ for all $s\in \bar\Delta$. Hence, via the map
        $$\mathcal{Z}=\{(t,x,s)\in  \Delta_{>r}\times X\times \bar\Delta\mid s(x)=t s_\infty(x)\}\xrightarrow{\pi}  \Delta_{>r}\times \bar\Delta$$
    that sends $(t,x,s)$ to $(t,s)$,
    the mixed Hodge module $\mathrm{R}^i\pi_*\bb{Q}^{\rr{H}}_{\cc{Z}}$ restricts to a VHS on $\bar{\Delta}^\circ\times \Delta\cros_{>r}$, denoted by $(\mathcal{H}^i,\nabla,F^p)$. In particular, the restriction of $\mathrm{R}^i\pi_*\bb{Q}^{\rr{H}}_{\cc{Z}}$ to $s\times \Delta_{>r}$ is $\mathrm{R}^ig_{s*}\bb{Q}\mid_{\Delta_{>r}}$.

     For each $0<\alpha\leq 1$, denote by $\cc{H}^i_{\alpha}$ the canonical extension of generalized eigenspaces of the monodromy operator of $\mathrm{R}^i\pi_*\bb{Q}^{\rr{H}}_{\cc{Z}}$ with respect to $\exp(-2\pi i \alpha)$ through the open inclusion $j\colon \Delta_{>r} \hookrightarrow \Delta_{>r}\cup \{\infty\}$, and $\nabla_\alpha$ the connection with logarithmic pole. Then, using the      version of the nilpotent orbit theorem in \cite[Thm.\,A]{DengNilpotentOrbit23}, we deduce that $F^p\cc{H}_\alpha^i:=j_*F^p\cap \cc{H}_\alpha^i$ are sub-vector bundles of $\cc{H}_\alpha^i$ on $\bar\Delta^\circ\times( \Delta_{>r}\cup\{\infty\})$.

    Since  
    \begin{equation}\label{eq:nearby-cycle-filtration}
        F^p\cc{H}^i_\alpha/\tfrac{1}{t}F^p\cc{H}^i_\alpha=F^p\psi_{1/t,,\exp(-2\pi i\alpha)}\rr{R}^i\pi_{*}\bb{Q},
    \end{equation}
    with $F^i$ being the Hodge filtration, we deduce that
        $$\begin{aligned}
            F^p(\cc{H}^i_\alpha/\tfrac{1}{t}\cc{H}_\alpha^i)|_{\{s\}\times \{\infty\}}
            &=F^p\cc{H}^i_\alpha|_{\{s\}\times \{\infty\}}/\tfrac{1}{t}F^p\cc{H}_\alpha^i|_{\{s\}\times \{\infty\}}\\
            &=F^p\psi_{1/t,\exp(-2\pi i\alpha)}\rr{R}^ig_{s*}\bb{Q}_{\fk{X}_s}\Hodge
        \end{aligned}$$
    for any $s\in S$. In particular, the Hodge numbers of $\psi_{1/t}\mathrm{R}^ig_{s*}\bb{Q}\Hodge$ are independent of $s$.
\end{proof}

\begin{prop}\label{prop:independence}
    For any rational number $\alpha$, $\lambda=\exp(2\pi i\alpha)$, and any $p\in \bb{Z}$, the dimensions of $\rr{gr}_F^p\psi_{1/t,\lambda}\Pi\,\rr{R}^r{f}_{s*}\mathbb{Q}^{\rr{H}}_{\tilde U} $ are independent of $f_s$. 
     
\end{prop}
\begin{proof}
We define
    $$\mathfrak{X}_S=\{(t,x,s)\in \bb{P}^1\times X\times S\mid ts_{\infty}(x)-s(x)=0\}$$
and get a diagram 
\begin{equation}
\begin{tikzcd}
    \fk{X}_S \ar[d,"g_S"]  & \tilde{U}_S =\tilde{g}_S\inv(\bb{A}^1_S)\ar[d,"\tilde{f}_S"] \ar[l,hook']  \\
    \bb{P}^1_S &\bb{A}^1_S\ar[l,hook']  
\end{tikzcd}
\end{equation}
such that the restriction to each $s\in S$ recovers the previous diagram  \eqref{eq:diagram-KKP}. Using the same argument as that in the proof of \cref{prop:distinguished-triangle}, we have a distinguished triangle
    $$
        \tilde{f}_{S*}\bb{Q}_{\tilde U_S}^{\rr{H}} \to \Pi\,  \tilde{f}_{S*}\bb{Q}_{\tilde U_S}^{\rr{H}} \to \rr{R}\Gamma(\tilde{U},\bb{Q}^{\rr{H}}_{\tilde U})\otimes_{\bb{C}}\bb{Q}_{\bb{A}^1_S}^{\rr{H}}\xrightarrow{+1}
    $$
    in $\rr{D}^b(\rr{MHM}(\bb{A}^1_S))$,
where $\Pi$ is the relative version of the functor defined in \eqref{eq:Pi}.
Applying the nearby cycle functor $\psi_{1/t}$, we have a long exact sequence
    $$
        \dots\to
        \psi_{1/t}\rr{R}^i\tilde{f}_{S*}\bb{Q}_{\tilde U_S}^{\rr{H}} 
        \xrightarrow{a_i}
        \psi_{1/t}\Pi\,  \rr{R}^i\tilde{f}_{S*}\bb{Q}_{\tilde U_S}^{\rr{H}} 
        \xrightarrow{b_i} 
        \rr{H}^i(\tilde{U},\bb{Q}^{\rr{H}}_{\tilde U}) \otimes \bb{Q}_S^{\rr{H}}
        \xrightarrow{c_i} 
        \psi_{1/t}\rr{R}^{i+1}\tilde{f}_{S*}\bb{Q}_{\tilde U_S}^{\rr{H}} 
        \to \cdots.
    $$

By \eqref{eq:nearby-cycle-filtration}, the mixed Hodge module $\psi_{1/t}\rr{R}^i\tilde{f}_{S*}\bb{Q}_{\tilde U_S}^{\rr{H}}$ are smooth on $S$. In particular, the morphisms $c_i$ are morphisms of variation of mixed Hodge structures, so that $\im(c_i)$ are again variation of mixed Hodge structures. So we have an exact sequence
$$
    0\to
    \rr{Im}(c_{i-1}) 
    \to
     \psi_{1/t}\rr{R}^i\tilde{f}_{S*}\bb{Q}_{\tilde U_S}^{\rr{H}} 
    \xrightarrow{a_i}
    \psi_{1/t}\Pi\,  \rr{R}^i\tilde{f}_{S*}\bb{Q}_{\tilde U_S}^{\rr{H}}  
    \xrightarrow{b_i}
    \rr{H}^i(\tilde{U},\bb{Q}^{\rr{H}}_{\tilde U}) \otimes \bb{Q}_S^{\rr{H}}
    \xrightarrow{c_i}
    \rr{Im}(c_{i }) 
    \to 0.
$$
Taking the fiber at $s\in S$, we  get an exact sequence
$$
    0\to
    \rr{Im}(c_{i-1})\mid_{s} 
    \to
     \psi_{1/t}\rr{R}^i\tilde{f}_{s*}\bb{Q}_{\tilde U}^{\rr{H}} 
    \xrightarrow{a_i}
    \psi_{1/t}\Pi\,  \rr{R}^i\tilde{f}_{s*}\bb{Q}_{\tilde U}^{\rr{H}}  
    \xrightarrow{b_i}
    \rr{H}^i(\tilde{U},\bb{Q}^{\rr{H}}_{\tilde U}) \otimes \bb{Q}^{\rr{H}}
    \xrightarrow{c_i}
    \rr{Im}(c_{i })\mid_{s} 
    \to 0.
$$

As $\rr{Im}(c_{i-1})$ and $\rr{Im}(c_i)$ are variation of mixed Hodge structures on $S$, the Hodge numbers of their fibers are independent of $s\in S$. By \cref{lem:independence-2}, the Hodge numbers of $\psi_{1/t}\rr{R}^i\tilde{f}_{s*}\bb{Q}_{\tilde U}^{\rr{H}}$ are also independent of $s\in S$. Hence, by the long exact sequence, the Hodge numbers of $\psi_{1/t}\Pi\,  \rr{R}^i\tilde{f}_{s*}\bb{Q}_{\tilde U}^{\rr{H}} $ are independent of $s\in S$.
\end{proof}

\begin{proof}[Proof of \upshape\cref{intro::non-degenerate}]
    
For any two non-degenerate functions $f_1$ and $f_2$ with respect to $(X,D)$, there exist two sections $s_1$ and $s_2$ in $\rr{H}^0(X,\cc{O}_X(D))$ such that $f_j=f_{s_j}$ for $j=1,2$. 
By \cref{cor:stationary-phase-functions}, one has
    \begin{equation}\label{eq:irregular-Hodge-max-ext}
        \begin{split}
            \dim \rr{gr}_{F_{\rr{irr}}}^{\alpha+p}\rr{H}^{i-n}_{\rr{dR}}(U,f_j) 
            &=\dim \mathrm{gr}_F^p\psi_{1/t,\lambda}\Pi\,\rr{R}^{i-n}f_{j*}\bb{Q}^{\rr{H}}_U\\ 
        \end{split}
    \end{equation}
for $j=1,2$, $\alpha\in [0,1)$, and $\lambda=\exp(2\pi i \lambda)$. By \cref{prop:independence}, we have
    $$\dim \mathrm{gr}_F^p\psi_{1/t,\lambda}\Pi\,\rr{R}^{i-n}f_{1*}\bb{Q}^{\rr{H}}_U=\dim \mathrm{gr}_F^p\psi_{1/t,\lambda}\Pi\,\rr{R}^{i-n}f_{2*}\bb{Q}^{\rr{H}}_U.$$ 
 Hence, we  conclude that 
    $$\dim \rr{gr}_{F_{\rr{irr}}}^{\alpha}\rr{H}^i_{\rr{dR}}(U,f_1)=\dim \rr{gr}_{F_{\rr{irr}}}^{\alpha}\rr{H}^i_{\rr{dR}}(U,f_2),$$
for any $\alpha\in \bb{Q}$.
\end{proof}

\section{An explicit formula for strongly non-degenerate functions}\label{sec:explicit-formula}
In this section, we give an explicit formula for the irregular Hodge numbers for a strongly non-degenerate function $f$ as in \cref{defn:non-degenerate} when the pole divisor is reduced.

\subsection{Spectrum polynomials}
\begin{defn}
    For a complex of $\bb{Q}$-filtered vector spaces $(\rr{V}^\bullet,F^\bullet)$, we define its Spectrum polynomial as
    $$
        \mathrm{Sp}_{(\rr{V}^\bullet,F^\bullet)}(t):=\sum_{k\in \bb{Z}}(-1)^k\sum_{\alpha \in \mathbb{Q}} \dim \mathrm{gr}_F^{\alpha} \rr{V}^k \cdot  t^{\alpha} \in \bb{Z}[t^{1/\infty}].
    $$
    When $\rr{V}^\bullet$ is $\rr{H}^*_{\rr{dR},c}(Y)$ for a variety $Y$, $\rr{H}^*_{\rr{dR},c}(U,f)$ or $\rr{R}\Gamma_c(f\inv(0), \psi_{ f}(\bb{Q}_{U}))$  for a regular function $f$ on $U$, we denote the corresponding Spectrum polynomials by $\mathrm{Sp}^c_Y(t)$, $\mathrm{Sp}^c_f(t)$ and $\mathrm{Sp}^c_{\psi_{f}}(t)$ respectively.
\end{defn}

\begin{prop}\label{proposition:unipotent-computation}
Assume that $(X,D,f)$ is a pair with a non-degenerate function with $D$ reduced and $f$ strongly non-degenerate.  Then we have
\begin{align*}
  \mathrm{Sp}_{f}^{c}(t)
  &= \mathrm{Sp}^c_X(t) + \sum_{k=1}^r (-1)^k \left\{ (1+t+\cdots + t^{k}) \sum_{|I|=k} \mathrm{Sp}^c_{D_I}(t) \right\} \\
  &\qquad - \sum_{k=1}^r (-1)^k \left\{ (t+t^2\cdots + t^{k}) \sum_{|I|=k}\mathrm{Sp}^c_{D_I\cap Z }(t) \right\}.
\end{align*}
\end{prop}

\begin{rmk}
    When $f$ is a strongly non-degenerate function, the only possibly non-vanishing cohomology is the middle one by \cref{cor:vanishing}. Hence, the irregular Hodge numbers of $\rr{H}^{\dim X}_{\rr{dR},c}(X-|D|,f)$ are given by the coefficients of $t^\alpha$ in  $(-1)^{\dim X}\mathrm{Sp}^c_{f}(t)$.
\end{rmk}

\begin{proof}[Proof of Proposition~\ref{proposition:unipotent-computation}]
Taking similar notation from the proof of \cref{lemma:snc-pair-implies-pure}. Let $D=\sum_{i=1}^r e_iD_i$ be a simple normal crossing divisor in $X$ and $f$ a non-degenerate function on $U:=X- D$ defined by $s_0$ and  $s_\infty$ in $\rr{H}^0(X,\cc{O}_X(D))$ as in \cref{defn:non-degenerate}. We denote by $B=Z(s_0)\cap D$ the base locus. Let
        $$\mathfrak{X} =\{(t,x)\in \bb{P}^1\times X\mid ts_{\infty}-s_0=0\}$$
and $g\colon \fk{X} \to \bb{P}^1$ be the projection to the first factor. 

Notice that $U$ is embedded in $\fk{X}$ via the graph embedding of $f$. We define $\bar{U}:=\fk{X}-\bb{P}^1\times B$, which contains $U$.  
We denote by the restrictions of $g$ to $\bar{U}$  by $\bar{f}$, then we have the following commutative diagram
\begin{equation}
\begin{tikzcd}
    \fk{X}  \ar[d,"g "]  & \bar{U}  =\fk{X} - (\bb{P}^1\times B)\ar[d,"\bar{f}"] \ar[l,hook'] & U\ar[d,"f "]\ar[l,hook']\\
    \bb{P}^1 &\bb{P}^1\ar[l,equal] &\bb{A}^1\ar[l,hook'].
\end{tikzcd}
\end{equation}

Let $j\colon D \setminus B \to D$ be the
open immersion. Notice that $\bar{f}\inv (\bb{A}^1)=U$. So $\phi_{1/f}(\mathbb{Q}_{U}^\rr{H})=\phi_{1/\bar{f}}(\mathbb{Q}_{U}^\rr{H})$.

\begin{lem}\label{lem:M-S-S}
    $\phi_{1/g}\bb{Q}^{\rr{H}}|_{\infty\times B  }=0$.
\end{lem}
\begin{proof}
    Since the functor $\rr{rat}$ sending mixed Hodge modules to their underlying perverse sheaves is faithful and exact, it suffices to show that $\phi_{1/g}\bb{Q}|_{\infty\times B}=0$. This can be deduced from \cite[Prop. 4.1]{MaximEtAlHirzebruchMilnor13} by setting $\cc{Y}=\fk{X}$, $X=P$, $X'$ is smooth hypersurface in the linear system $|P|$, and $X''=X\cap X'=B$ is the base locus.
\end{proof}
By \cref{lem:M-S-S}, the natural map
\[
    j_{!}\phi_{1/f}(\mathbb{Q}_{U}^\rr{H}) = j_{!}j^{\ast}\phi_{1/g}(\mathbb{Q}^{\rr{H}}_{\fk{X}})
    \to \phi_{1/g}(\mathbb{Q}^{\rr{H}}_{\fk{X}})
\]
is an isomorphism.  Therefore, the following diagram
\[
    \begin{tikzcd}
        j_{!}\mathbb{Q}^{\rr{H}}_{D\setminus B} \ar[r] \ar[d] & j_{!}\psi_{1/f}(\mathbb{Q}^{\rr{H}}_{U}) \ar[d] \\
        \mathbb{Q}^{\rr{H}}_{D} \ar[r] & \psi_{1/g}(\mathbb{Q}^{\rr{H}}_{\fk{X}})
    \end{tikzcd}
\]
is a homotopy Cartesian diagram.  Since the cone of the left vertical arrow is $i_{\ast}\mathbb{Q}^{\rr{H}}_{B}$, where $i\colon B \to D$ is the inclusion morphism, we get the following exact triangle of mixed Hodge modules:
\[
    j_{!} \psi_{1/f,1}(\mathbb{Q}^{\rr{H}}_U) \to \psi_{1/g,1}(\mathbb{Q}^{\rr{H}}_{\fk{X}}) \to i_{\ast}\mathbb{Q}^{\rr{H}}_{B}.
\]
Since the Spectrum polynomial is an additive invariant, we have
\begin{equation*}\label{eq:additive-spectrum}
    \mathrm{Sp}_{\psi_{1/f,1}}^c(t) = \mathrm{Sp}^c_{\psi_{1/g,1}}(t) - \mathrm{Sp}^c_B(t).
\end{equation*}

For each non-empty subset $I\subset\{1,\dots,r\}$, set
$ D_I:=\bigcap_{i\in I}D_i,$
and, by convention, $D_{\emptyset}:=X$.  Put $Z:=Z(s_0)$ and, for $I\neq\emptyset$, denote $Z_I:=D_I\cap Z$. For $I \neq \emptyset$, the morphism
    $\overline{f}\colon \mathfrak{X} \to \mathbb{P}^1$ 
is étale locally of the form 
    $$(x_1,\ldots,x_n) \mapsto \prod_{i\in I}x^{e_{i}}_{i}$$
around any point $x \in D^{\circ}_I\setminus D_I^{\circ} \cap B$.
So by \cite[Expos\'e I, Th\'eorem\`e~3.3]{sga7}, we have
\[
\mathcal{H}^i\psi_{f^{-1},1}(\mathbb{Q}_{\mathfrak{X}}^{\rr{H}})|_{D_I^{\circ} \setminus B} \simeq \mathbb{Q}_{D^{\circ}_I\setminus B}^{\rr{H},\oplus \binom{|I|}{i}}(-i).
\]
Hence, for any nonempty $I \subset \{1,\ldots,r\}$, the Spectrum polynomial of
$\psi_{1/f,1}(\mathbb{Q}^{\rr{H}}_{U})|_{D_I^{\circ}\setminus B}$ equals
\[
(1-t)^{|I|-1}\cdot (\mathrm{Sp}^{c}_{D_I^{\circ}}(t) - \mathrm{Sp}_{D^{\circ}_I \cap Z}^{c}(t)).
\]
By additivity, we find that
\[
\mathrm{Sp}_{\psi_{1/f,1}}^{c}(t) = \sum_{\substack{I\subset\{1,\ldots,r\}\\ I\neq \emptyset}}
(1-t)^{|I|-1} \cdot (\mathrm{Sp}^{c}_{D_I^{\circ}}(t) - \mathrm{Sp}_{D^{\circ}_I \cap Z}^{c}(t)).
\]
By inclusion-exclusion, we have
\begin{align*}
  \mathrm{Sp}^c_{D_I^{\circ}}(t) &= \sum_{J\supset I} (-1)^{|J|-|I|} \mathrm{Sp}^c_{D_J}(t) \\
  \mathrm{Sp}^c_{D_I^{\circ} \cap Z}(t) &= \sum_{J\supset I} (-1)^{|J|-|I|} \mathrm{Sp}^c_{D_J\cap Z}(t).
\end{align*}
Therefore,
\begin{align*}
  \mathrm{Sp}_{\psi_{1/f,1}}^{c}(t)
  &= \sum_{\substack{I\subset\{1,\ldots,r\}\\ I\neq \emptyset}} \sum_{J\supset I} (-1)^{|J|-|I|}(1-t)^{|I|-1} \left( \mathrm{Sp}^c_{D_J}(t) - \mathrm{Sp}^c_{D_J\cap Z}(t) \right) \\
  &= \sum_{\substack{J\subset\{1,\ldots,r\}\\ J\neq \emptyset}} (-1)^{|J|-1} \frac{1}{t-1}
  \left( \mathrm{Sp}^c_{D_J}(t) - \mathrm{Sp}^c_{D_J\cap Z}(t) \right)\sum_{\substack{I \subset J \\ I\neq \emptyset}}
  (t-1)^{|I|} \\
  &= \sum_{k=1}^r (-1)^{k-1} (1+t+\cdots + t^{k-1}) \sum_{|J|=k} \left( \mathrm{Sp}^c_{D_J}(t) - \mathrm{Sp}^c_{D_J\cap Z}(t) \right).
\end{align*}
Finally, by Theorem~\ref{thm:KKP-numbers} and
Lemma~\ref{lemma:snc-pair-implies-pure}, we have
\begin{align*}
  \mathrm{Sp}_{f,1}(t)
  &= \mathrm{Sp}_U^c(t) - t\mathrm{Sp}_{\psi_{1/f,1}}^{c}(t) \\
  &= \mathrm{Sp}^c_X(t) - \sum_{\substack{J\subset\{1,\ldots,r\}\\ J\neq \emptyset}} (-1)^{|J|-1} \mathrm{Sp}^c_{D_{J}}(t) - t \mathrm{Sp}_{\psi_{1/f,1}}^c(t)\\
  &= \mathrm{Sp}^c_X(t) - \sum_{k=1}^r(-1)^{k-1}(1 + t + \cdots + t^k) \sum_{|J|=k} \mathrm{Sp}^c_{D_J}(t) \\
  &\qquad+ \sum_{k=1}^r(-1)^{k-1}t(1 + \cdots + t^{k-1})\mathrm{Sp}^c_{D_J\cap Z}(t).
\end{align*}
This completes the proof.
\end{proof}

\subsection{Examples}
\begin{exe}\label{example:calculation}
    Back to the setting of \cref{example}, we have $X=\bb{P}^2$ and $D=L_0\cup L_1\cup L_2$, where the three lines $L_0, L_1$ and $L_2$ are defined by $x_0,x_1$, and $x_2$ respectively. In particular, the interior $U:=X -|D|$ is the torus $\bb{G}_{m}^2$. We take the (strongly) non-degenerate functions $f$ on $(X,D)$ of the form
        $$\frac{F}{x_0x_1x_2}\colon U\to \bb{A}^1$$
    where $F$ are cubic homogeneous polynomials in $x_0,x_1$, and $x_2$, such that $Z(F)$ is transverse to $D$.
    \begin{itemize}
        \item The Spectrum polynomial of $X=\bb{P}^2$ is
            $$\mathrm{Sp}^c_X(t)=1+t+t^2.$$
        \item For each $i=0,1,2$, the divisor $D_i=L_i\simeq \bb{P}^1$ has Spectrum polynomial
            $$\mathrm{Sp}^c_{D_i}(t)=1+t.$$
        \item For each $0\leq i<j\leq 2$, the intersection $D_i\cap D_j$ is a point, whose Spectrum polynomial is
            $$\mathrm{Sp}^c_{D_i\cap D_j}(t)=1.$$
        \item Finally, the intersection $Z\cap D_i$ consists of three points for each $i=0,1,2$, hence
            $$\mathrm{Sp}^c_{D_i\cap Z}(t)=3.$$
    \end{itemize}
    Applying \cref{proposition:unipotent-computation}, we have
    \begin{align*}
      \mathrm{Sp}_{f}^{c}(t) = 1 + 7t +  t^2 .
    \end{align*}
    Hence, the irregular Hodge numbers of $\rr{H}^2_{\rr{dR},c}(U,f)\simeq \rr{H}^2_{\rr{dR}}(U,f)$ are given by $h^{0}_{\rr{irr}}(U,f,2)=1$, $h^{1}_{\rr{irr}}(U,f,2)=7$, and $h^{2}_{\rr{irr}}(U,f,2)=1$.
\end{exe}

\begin{exe}\label{example:calculation2}
    Back to the setting of \cref{example2}, we have $X=\bb{P}^2$ and $D=Z(x_0^3+x_1^3+x_2^3+3x_0x_1x_2)$. We take the (strongly) non-degenerate functions $f$ on $(X,D)$ of the form
        $$\frac{x_0x_1x_2}{x_0^3+x_1^3+x_2^3+3x_0x_1x_2}\colon U\to \bb{A}^1.$$
    \begin{itemize}
        \item The Spectrum polynomial of $X=\bb{P}^2$ is
            $$\mathrm{Sp}^c_X(t)=1+t+t^2.$$
        \item The divisor $D$ is an elliptic curve, having Spectrum polynomial
            $$\mathrm{Sp}^c_{D}(t)=1-(1+t)+t=0.$$ 
        \item Finally, the intersection $Z\cap D$ consists of $9$ points, hence
            $$\mathrm{Sp}^c_{D\cap Z}(t)=9.$$
    \end{itemize}
    Applying \cref{proposition:unipotent-computation} to $f+c$ for a generic $c\in \bb{C}$, we have
    \begin{align*}
      \mathrm{Sp}_{f}^{c}(t) = 1 + 10t +  t^2 .
    \end{align*}
    Hence, the irregular Hodge numbers of $\rr{H}^2_{\rr{dR},c}(U,f)\simeq \rr{H}^2_{\rr{dR}}(U,f)$ are given by $h^{0}_{\rr{irr}}(U,f,2)=1$, $h^{1}_{\rr{irr}}(U,f,2)=10$, and $h^{2}_{\rr{irr}}(U,f,2)=1$.
\end{exe}

\section*{Acknowledgement}
The authors thank Javier Fres\'an, Takuro Mochizuki, Claude Sabbah and  Christian Sevenheck and  for helpful discussions and comments.

{\footnotesize 
\bibliographystyle{alpha} 
\bibliography{bibtex} }

\end{document}